\documentclass[11pt]{article}\UseRawInputEncoding
\usepackage{amssymb,amsfonts,amsmath,latexsym,epsf,tikz,url}
\usepackage[usenames,dvipsnames]{pstricks}
\usepackage{pstricks-add}
\usepackage{epsfig}
\usepackage{pst-grad} 
\usepackage{pst-plot} 
\usepackage[space]{grffile} 
\usepackage{etoolbox} 
\makeatletter 
\patchcmd\Gread@eps{\@inputcheck#1 }{\@inputcheck"#1"\relax}{}{}
\makeatother

\newtheorem{theorem}{Theorem}[section]

\newtheorem{corollary}[theorem]{Corollary}

\newcommand{\proof}{\noindent{\bf Proof.\ }}
\newcommand{\qed}{\hfill $\square$\medskip}
\newcommand{\gst}{\gamma_{\rm st}}

\begin{document}

\title{Strong domination number of a modified graph }

\author{ Saeid Alikhani$^{1,}$\footnote{Corresponding author}
	\and
Nima Ghanbari$^{2}$
}

\date{\today}

\maketitle

\begin{center}
	
$^1$Department of Mathematical Sciences, Yazd University, 89195-741, Yazd, Iran\\
$^2$Department of Informatics, University of Bergen, P.O. Box 7803, 5020 Bergen, Norway\\

\medskip
\medskip
{\tt  $^{1}$alikhani@yazd.ac.ir 
~~
$^{2}$Nima.Ghanbari@uib.no
 }

\end{center}

\begin{abstract}
 Let $G=(V,E)$ be a simple graph. A set $D\subseteq V$ is a strong dominating set of $G$, if for every vertex $x\in V\setminus D$ there is a vertex $y\in D$ with $xy\in E(G)$ and $\deg(x)\leq \deg(y)$. The strong domination number $\gst(G)$ is defined as the minimum cardinality of a strong dominating set.  In this paper, we study the effects on $\gst(G)$ when $G$ is modified by operations on vertex and edge of $G$. 
\end{abstract}


\medskip
\noindent{\bf AMS Subj.\ Class.: } 05C15, 05C25.

\section{Introduction}

A dominating set of a graph $G=(V,E)$ is a subset $D$ of $V$ such that every vertex in $V\setminus D$ is adjacent to at least one member of $D$.  
The minimum cardinality of all dominating sets of $G$ is called  the  domination number of $G$ and is denoted by $\gamma(G)$. This parameter has  been extensively studied in the literature and there are  hundreds of papers concerned with domination.  
For a detailed treatment of domination theory, the reader is referred to \cite{domination}. Also, the concept of domination and related invariants have
been generalized in many ways.

A set $D\subseteq V(G)$ is a strong dominating set of $G$, if for every vertex $x\in V(G)\setminus D$ there is a vertex $y\in D$ with $xy\in E(G)$ and $deg(x)\leq deg(y)$. The strong
domination number $\gamma_{st}(G)$ is defined as the minimum cardinality of a strong dominating set. 
A strong dominating set with cardinality $\gamma_{st}(G)$ is called a $\gamma_{st}$-set.
The strong domination number was introduced in \cite{DM} and some upper bounds on this parameter presented in \cite{DM2,DM}. Similar to strong domination number, a set $D\subset V$  is a weak  dominating set of $G$, if every vertex $v\in V\setminus S$  is
adjacent to a vertex $u\in D$ such that $deg(v)\geq deg(u)$ (see \cite{Boutrig}). The minimum cardinality of a weak dominating set of $G$ is denoted by $\gamma_w(G)$. Boutrig and  Chellali proved that the relation $\gamma_w(G)+\frac{3}{\Delta+1}\gamma_{st}(G)\leq n$ holds for any connected graph of order $n\geq 3.$ Alikhani, Ghanbari and Zaherifard \cite{sub} examined the effects on $\gamma_{st}(G)$ when $G$ is modified by  edge deletion, edge subdivision and edge contraction. Also they studied the strong domination number of $k$-subdivision of $G$.   

Motivated by counting of the number of dominating sets of a graph and domination polynomial (see e.g. \cite{euro,saeid1}), the number of the strong dominating sets for certain
graphs has studied in \cite{JAS}.  

Let $e$ be an edge of a connected simple graph $G$. The graph obtained by removing  an edge $e$ from $G$ is denoted by $G-e$.
The edge subdivision operation for an edge $\{u,v\}\in E$ is the deletion of $\{u,v\}$ from $G$ and the addition of two edges $\{u,w\}$ and $\{w,v\}$ along with the new vertex $w$. 
A graph which has been derived from $G$ by deleting a vertex $v$ is denoted by $G-v$. The contraction 
of $v$ in $G$ denoted by $G/v$ is the graph obtained by deleting $v$ and putting 
a clique on the (open) neighbourhood of $v$.    
An edge contraction is an operation that removes an edge from a graph while simultaneously merging the two vertices that it previously joined. The resulting induced graph is written as $G/e$.

In this paper,   we examine the effects on $\gamma_{st}(G)$ when $G$ is modified by operations such as vertex deletion, vertex contraction and  edge contraction

\section{Main Results}

 In this section, we study  the effects on $\gamma_{st}(G)$ when $G$ is modified by some operations. First we consider the vertex deletion.

	\begin{theorem}\label{thm:G-v}
 If $G=(V,E)$ is  a graph and $v\in V$, then,
 $$\gst (G)-\deg(v)\leq\gst(G-v)\leq \gst (G)+\deg(v)-1.$$
 Furthermore, these bounds are tight.
	\end{theorem}

	\begin{proof}
First we consider the upper bound. Suppose that $D$ is a $\gst$-set of $G$. 
If $v\in D$, then $\left(D\cup N(v)\right)\setminus\{v\}$ is a strong dominating set of $G-v$ and we are done. 
If $v\notin D$, then there exists $u\in N(v)$ such that $u$ is strong dominate  $v$. So $D\cup N(v)$ is a strong dominating set with size at most $\gst (G)+\deg(v)-1$.
Therefore we have $\gst(G-v)\leq \gst (G)+\deg(v)-1$. The equality holds for the star graph, and $v$ is the universal vertex. Now, we obtain  the lower bound. First we consider  $G-v$ and suppose that $S$ is a $\gst$-set of $G-v$. We have two cases for $v$ in $G$:
\begin{itemize}
\item[(i)]
$\deg (v)> \deg (u)$ for all $u\in N(v)$. Then clearly $S\cup\{v\}$ is a strong dominating set for $G$. So $\gst (G)\leq\gst(G-v)+1$.

\item[(ii)]
There exists $u\in N(v)$ such that $\deg(u)\geq \deg(v)$. So  $S\cup N(v)$ is a strong dominating set for $G$ and so $\gst (G)\leq\gst(G-v)+\deg (v)$.
\end{itemize}
Therefore we have $\gst(G-v)\geq \gst (G)-\deg(v)$. Now, we show that this bound is tight. Consider Figure \ref{fig:G-v-lower}. The set of black vertices is a $\gst$-set of $G$, say $D$. Now, $D\setminus N(v)$ is a $\gst$-set of $G-v$, and we have the result.
\qed
	\end{proof}

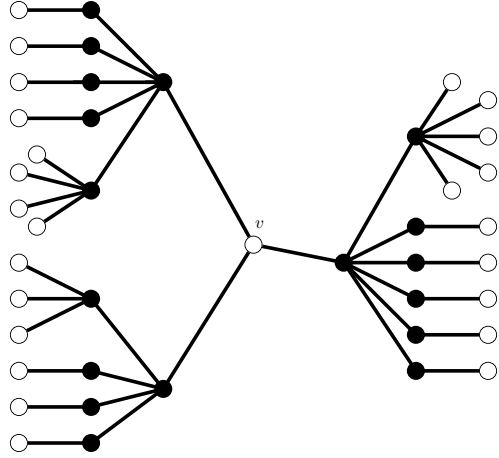
\begin{figure}
\begin{center}
\psscalebox{0.6 0.6}
{
\begin{pspicture}(0,-6.8)(10.802778,3.202778)
\psline[linecolor=black, linewidth=0.08](9.001389,0.20138916)(10.601389,1.0013891)(10.601389,1.0013891)
\psline[linecolor=black, linewidth=0.08](9.001389,0.20138916)(10.601389,0.20138916)(10.601389,0.20138916)
\psline[linecolor=black, linewidth=0.08](9.001389,0.20138916)(10.601389,-0.5986108)(10.601389,-0.5986108)
\psline[linecolor=black, linewidth=0.08](9.001389,0.20138916)(9.801389,-0.99861085)(9.801389,-0.99861085)
\psline[linecolor=black, linewidth=0.08](9.001389,0.20138916)(9.801389,1.4013891)(9.801389,1.4013891)
\psline[linecolor=black, linewidth=0.08](9.001389,-1.7986108)(10.601389,-1.7986108)(10.601389,-1.7986108)
\psline[linecolor=black, linewidth=0.08](9.001389,-2.5986109)(10.601389,-2.5986109)(10.601389,-2.5986109)
\psline[linecolor=black, linewidth=0.08](9.001389,-3.3986108)(10.601389,-3.3986108)(10.601389,-3.3986108)
\psline[linecolor=black, linewidth=0.08](9.001389,-4.198611)(10.601389,-4.198611)(10.601389,-4.198611)
\psline[linecolor=black, linewidth=0.08](9.001389,-4.998611)(10.601389,-4.998611)(10.601389,-4.998611)
\psline[linecolor=black, linewidth=0.08](7.4013886,-2.5986109)(9.001389,0.20138916)(9.001389,0.20138916)
\psline[linecolor=black, linewidth=0.08](7.4013886,-2.5986109)(9.001389,-1.7986108)(9.001389,-1.7986108)
\psline[linecolor=black, linewidth=0.08](7.4013886,-2.5986109)(9.001389,-2.5986109)(9.001389,-2.5986109)
\psline[linecolor=black, linewidth=0.08](7.4013886,-2.5986109)(9.001389,-3.3986108)(9.001389,-3.3986108)
\psline[linecolor=black, linewidth=0.08](7.4013886,-2.5986109)(9.001389,-4.198611)(9.001389,-4.198611)
\psline[linecolor=black, linewidth=0.08](7.4013886,-2.5986109)(9.001389,-4.998611)(9.001389,-4.998611)
\psline[linecolor=black, linewidth=0.08](0.20138885,3.0013893)(1.8013889,3.0013893)(1.8013889,3.0013893)
\psline[linecolor=black, linewidth=0.08](0.20138885,2.201389)(1.8013889,2.201389)(1.8013889,2.201389)
\psline[linecolor=black, linewidth=0.08](0.20138885,1.4013891)(1.8013889,1.4013891)(1.8013889,1.4013891)
\psline[linecolor=black, linewidth=0.08](0.20138885,0.60138917)(1.8013889,0.60138917)(1.8013889,0.60138917)
\psline[linecolor=black, linewidth=0.08](3.401389,1.4013891)(1.4013889,1.4013891)(1.4013889,1.4013891)
\psline[linecolor=black, linewidth=0.08](1.8013889,2.201389)(3.401389,1.4013891)(3.0013888,1.4013891)
\psline[linecolor=black, linewidth=0.08](1.8013889,3.0013893)(3.401389,1.4013891)(3.401389,1.4013891)
\psline[linecolor=black, linewidth=0.08](3.401389,1.4013891)(1.8013889,0.60138917)(1.8013889,0.60138917)
\psline[linecolor=black, linewidth=0.08](0.20138885,-0.5986108)(1.8013889,-0.99861085)(1.8013889,-0.99861085)
\psline[linecolor=black, linewidth=0.08](1.8013889,-0.99861085)(0.6013889,-0.19861084)(0.6013889,-0.19861084)
\psline[linecolor=black, linewidth=0.08](0.20138885,-1.3986108)(1.8013889,-0.99861085)(1.8013889,-0.99861085)
\psline[linecolor=black, linewidth=0.08](1.8013889,-0.99861085)(0.6013889,-1.7986108)(0.6013889,-1.7986108)
\psline[linecolor=black, linewidth=0.08](1.8013889,-0.99861085)(3.401389,1.4013891)(3.401389,1.4013891)
\psline[linecolor=black, linewidth=0.08](0.20138885,-2.5986109)(1.8013889,-3.3986108)(0.20138885,-3.3986108)(0.20138885,-3.3986108)
\psline[linecolor=black, linewidth=0.08](0.20138885,-4.198611)(1.8013889,-3.3986108)(1.8013889,-3.3986108)
\psline[linecolor=black, linewidth=0.08](0.20138885,-4.998611)(1.8013889,-4.998611)(1.8013889,-4.998611)
\psline[linecolor=black, linewidth=0.08](0.20138885,-5.7986107)(1.8013889,-5.7986107)(1.8013889,-5.7986107)
\psline[linecolor=black, linewidth=0.08](0.20138885,-6.598611)(1.8013889,-6.598611)(1.8013889,-6.598611)
\psline[linecolor=black, linewidth=0.08](1.8013889,-3.3986108)(3.401389,-5.398611)(1.8013889,-4.998611)(1.8013889,-4.998611)
\psline[linecolor=black, linewidth=0.08](1.8013889,-5.7986107)(3.401389,-5.398611)(3.401389,-5.398611)
\psline[linecolor=black, linewidth=0.08](3.401389,-5.398611)(1.8013889,-6.598611)(1.8013889,-6.598611)
\psdots[linecolor=black, dotsize=0.4](1.8013889,-0.99861085)
\psdots[linecolor=black, dotsize=0.4](1.8013889,0.60138917)
\psdots[linecolor=black, dotsize=0.4](1.8013889,1.4013891)
\psdots[linecolor=black, dotsize=0.4](1.8013889,2.201389)
\psdots[linecolor=black, dotsize=0.4](1.8013889,3.0013893)
\psdots[linecolor=black, dotsize=0.4](1.8013889,-3.3986108)
\psdots[linecolor=black, dotsize=0.4](1.8013889,-4.998611)
\psdots[linecolor=black, dotsize=0.4](1.8013889,-5.7986107)
\psdots[linecolor=black, dotsize=0.4](1.8013889,-6.598611)
\psdots[linecolor=black, dotsize=0.4](9.001389,0.20138916)
\psdots[linecolor=black, dotsize=0.4](9.001389,-1.7986108)
\psdots[linecolor=black, dotsize=0.4](9.001389,-2.5986109)
\psdots[linecolor=black, dotsize=0.4](9.001389,-3.3986108)
\psdots[linecolor=black, dotsize=0.4](9.001389,-4.198611)
\psdots[linecolor=black, dotsize=0.4](9.001389,-4.998611)
\psline[linecolor=black, linewidth=0.08](5.4013886,-2.1986108)(3.401389,1.4013891)(3.401389,1.4013891)
\psline[linecolor=black, linewidth=0.08](3.401389,-5.398611)(5.4013886,-2.1986108)(5.4013886,-2.1986108)
\psline[linecolor=black, linewidth=0.08](7.4013886,-2.5986109)(5.4013886,-2.1986108)(5.4013886,-2.1986108)
\psdots[linecolor=black, dotsize=0.4](7.4013886,-2.5986109)
\psdots[linecolor=black, dotsize=0.4](3.401389,1.4013891)
\psdots[linecolor=black, dotsize=0.4](3.401389,-5.398611)
\psdots[linecolor=black, dotstyle=o, dotsize=0.4, fillcolor=white](9.801389,1.4013891)
\psdots[linecolor=black, dotstyle=o, dotsize=0.4, fillcolor=white](10.601389,1.0013891)
\psdots[linecolor=black, dotstyle=o, dotsize=0.4, fillcolor=white](10.601389,0.20138916)
\psdots[linecolor=black, dotstyle=o, dotsize=0.4, fillcolor=white](10.601389,-0.5986108)
\psdots[linecolor=black, dotstyle=o, dotsize=0.4, fillcolor=white](9.801389,-0.99861085)
\psdots[linecolor=black, dotstyle=o, dotsize=0.4, fillcolor=white](10.601389,-1.7986108)
\psdots[linecolor=black, dotstyle=o, dotsize=0.4, fillcolor=white](10.601389,-2.5986109)
\psdots[linecolor=black, dotstyle=o, dotsize=0.4, fillcolor=white](10.601389,-3.3986108)
\psdots[linecolor=black, dotstyle=o, dotsize=0.4, fillcolor=white](10.601389,-4.198611)
\psdots[linecolor=black, dotstyle=o, dotsize=0.4, fillcolor=white](10.601389,-4.998611)
\psdots[linecolor=black, dotstyle=o, dotsize=0.4, fillcolor=white](5.4013886,-2.1986108)
\psdots[linecolor=black, dotstyle=o, dotsize=0.4, fillcolor=white](0.20138885,3.0013893)
\psdots[linecolor=black, dotstyle=o, dotsize=0.4, fillcolor=white](0.20138885,2.201389)
\psdots[linecolor=black, dotstyle=o, dotsize=0.4, fillcolor=white](0.20138885,1.4013891)
\psdots[linecolor=black, dotstyle=o, dotsize=0.4, fillcolor=white](0.20138885,0.60138917)
\psdots[linecolor=black, dotstyle=o, dotsize=0.4, fillcolor=white](0.6013889,-0.19861084)
\psdots[linecolor=black, dotstyle=o, dotsize=0.4, fillcolor=white](0.20138885,-0.5986108)
\psdots[linecolor=black, dotstyle=o, dotsize=0.4, fillcolor=white](0.20138885,-1.3986108)
\psdots[linecolor=black, dotstyle=o, dotsize=0.4, fillcolor=white](0.6013889,-1.7986108)
\psdots[linecolor=black, dotstyle=o, dotsize=0.4, fillcolor=white](0.20138885,-2.5986109)
\psdots[linecolor=black, dotstyle=o, dotsize=0.4, fillcolor=white](0.20138885,-3.3986108)
\psdots[linecolor=black, dotstyle=o, dotsize=0.4, fillcolor=white](0.20138885,-4.198611)
\psdots[linecolor=black, dotstyle=o, dotsize=0.4, fillcolor=white](0.20138885,-4.998611)
\psdots[linecolor=black, dotstyle=o, dotsize=0.4, fillcolor=white](0.20138885,-5.7986107)
\psdots[linecolor=black, dotstyle=o, dotsize=0.4, fillcolor=white](0.20138885,-6.598611)
\rput[bl](5.441389,-1.8386109){$v$}
\end{pspicture}
}
\end{center}
\caption{\small{Graph $G$ which shows the tightness of the lower bound of Theorem \ref{thm:G-v}}}\label{fig:G-v-lower}
\end{figure}


The following theorem gives bounds for the strong domination number of a graph $G/v$, where $G/v$ is a graph obtained by $G$ and contraction of a vertex $v$. We recall that a vertex $v$ is a pendant vertex, if ${\rm deg}(v)=1$.  

	\begin{theorem}\label{thm:G/v-nonpendant}
If  $G=(V,E)$ is a graph and $v\in V$ is not a pendant vertex, then,
 $$\gst (G)-\deg (v)+1\leq \gst(G/v)\leq \gst (G)+1.$$
  Furthermore, these bounds are tight.
	\end{theorem}

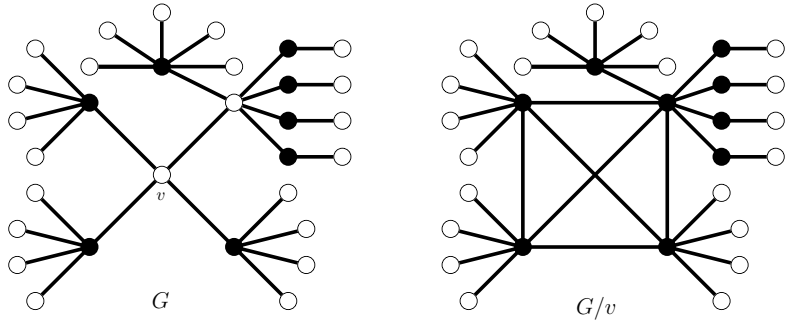
\begin{figure}
\begin{center}
\psscalebox{0.6 0.6}
{
\begin{pspicture}(0,-6.109306)(17.202778,0.91208345)
\psline[linecolor=black, linewidth=0.08](3.401389,-2.8893054)(1.8013889,-1.2893054)(1.8013889,-1.2893054)
\psline[linecolor=black, linewidth=0.08](3.401389,-2.8893054)(5.001389,-1.2893054)(5.001389,-1.2893054)
\psline[linecolor=black, linewidth=0.08](3.401389,-2.8893054)(5.001389,-4.4893055)(5.001389,-4.4893055)
\psline[linecolor=black, linewidth=0.08](3.401389,-2.8893054)(1.8013889,-4.4893055)(1.8013889,-4.4893055)
\psline[linecolor=black, linewidth=0.08](5.001389,-4.4893055)(6.201389,-3.2893054)(6.201389,-3.2893054)
\psline[linecolor=black, linewidth=0.08](5.001389,-4.4893055)(6.601389,-4.0893054)(6.601389,-4.0893054)
\psline[linecolor=black, linewidth=0.08](5.001389,-4.4893055)(6.601389,-4.8893056)(6.601389,-4.8893056)
\psline[linecolor=black, linewidth=0.08](5.001389,-4.4893055)(6.201389,-5.6893053)(6.201389,-5.6893053)
\psline[linecolor=black, linewidth=0.08](1.8013889,-4.4893055)(0.6013889,-5.6893053)(0.6013889,-5.6893053)
\psline[linecolor=black, linewidth=0.08](0.20138885,-4.8893056)(1.8013889,-4.4893055)(1.8013889,-4.4893055)
\psline[linecolor=black, linewidth=0.08](0.20138885,-4.0893054)(1.8013889,-4.4893055)(1.8013889,-4.4893055)
\psline[linecolor=black, linewidth=0.08](0.6013889,-3.2893054)(1.8013889,-4.4893055)(1.8013889,-4.4893055)
\psline[linecolor=black, linewidth=0.08](1.8013889,-1.2893054)(0.6013889,-0.08930542)(0.6013889,-0.08930542)
\psline[linecolor=black, linewidth=0.08](0.20138885,-0.8893054)(1.8013889,-1.2893054)(1.8013889,-1.2893054)
\psline[linecolor=black, linewidth=0.08](0.20138885,-1.6893054)(1.8013889,-1.2893054)(1.8013889,-1.2893054)
\psline[linecolor=black, linewidth=0.08](0.6013889,-2.4893055)(1.8013889,-1.2893054)(1.8013889,-1.2893054)
\psline[linecolor=black, linewidth=0.08](5.001389,-1.2893054)(6.201389,-2.4893055)(6.201389,-2.4893055)
\psline[linecolor=black, linewidth=0.08](5.001389,-1.2893054)(6.201389,-1.6893054)(6.201389,-1.6893054)
\psline[linecolor=black, linewidth=0.08](5.001389,-1.2893054)(6.201389,-0.8893054)(6.201389,-0.8893054)
\psline[linecolor=black, linewidth=0.08](5.001389,-1.2893054)(6.201389,-0.08930542)(6.201389,-0.08930542)
\psline[linecolor=black, linewidth=0.08](5.001389,-1.2893054)(3.401389,-0.4893054)(3.401389,-0.4893054)
\psline[linecolor=black, linewidth=0.08](6.201389,-0.08930542)(7.4013886,-0.08930542)(7.4013886,-0.08930542)
\psline[linecolor=black, linewidth=0.08](6.201389,-0.8893054)(7.4013886,-0.8893054)(7.4013886,-0.8893054)
\psline[linecolor=black, linewidth=0.08](6.201389,-1.6893054)(7.4013886,-1.6893054)(7.4013886,-1.6893054)
\psline[linecolor=black, linewidth=0.08](6.201389,-2.4893055)(7.4013886,-2.4893055)(7.4013886,-2.4893055)
\psline[linecolor=black, linewidth=0.08](3.401389,-0.4893054)(1.8013889,-0.4893054)(3.401389,-0.4893054)(5.001389,-0.4893054)(5.001389,-0.4893054)
\psline[linecolor=black, linewidth=0.08](3.401389,-0.4893054)(4.601389,0.31069458)(4.601389,0.31069458)
\psline[linecolor=black, linewidth=0.08](3.401389,-0.4893054)(2.2013888,0.31069458)(2.2013888,0.31069458)
\psline[linecolor=black, linewidth=0.08](3.401389,-0.4893054)(3.401389,0.71069455)(3.401389,0.71069455)
\psdots[linecolor=black, dotsize=0.4](1.8013889,-4.4893055)
\psdots[linecolor=black, dotsize=0.4](1.8013889,-1.2893054)
\psdots[linecolor=black, dotsize=0.4](5.001389,-4.4893055)
\psdots[linecolor=black, dotsize=0.4](6.201389,-0.08930542)
\psdots[linecolor=black, dotsize=0.4](6.201389,-0.8893054)
\psdots[linecolor=black, dotsize=0.4](6.201389,-1.6893054)
\psdots[linecolor=black, dotsize=0.4](6.201389,-2.4893055)
\psdots[linecolor=black, dotsize=0.4](3.401389,-0.4893054)
\psdots[linecolor=black, dotstyle=o, dotsize=0.4, fillcolor=white](7.4013886,-0.08930542)
\psdots[linecolor=black, dotstyle=o, dotsize=0.4, fillcolor=white](7.4013886,-0.8893054)
\psdots[linecolor=black, dotstyle=o, dotsize=0.4, fillcolor=white](7.4013886,-1.6893054)
\psdots[linecolor=black, dotstyle=o, dotsize=0.4, fillcolor=white](7.4013886,-2.4893055)
\psdots[linecolor=black, dotstyle=o, dotsize=0.4, fillcolor=white](6.201389,-3.2893054)
\psdots[linecolor=black, dotstyle=o, dotsize=0.4, fillcolor=white](6.601389,-4.0893054)
\psdots[linecolor=black, dotstyle=o, dotsize=0.4, fillcolor=white](6.601389,-4.8893056)
\psdots[linecolor=black, dotstyle=o, dotsize=0.4, fillcolor=white](6.201389,-5.6893053)
\psdots[linecolor=black, dotstyle=o, dotsize=0.4, fillcolor=white](0.6013889,-5.6893053)
\psdots[linecolor=black, dotstyle=o, dotsize=0.4, fillcolor=white](0.20138885,-4.8893056)
\psdots[linecolor=black, dotstyle=o, dotsize=0.4, fillcolor=white](0.20138885,-4.0893054)
\psdots[linecolor=black, dotstyle=o, dotsize=0.4, fillcolor=white](0.6013889,-3.2893054)
\psdots[linecolor=black, dotstyle=o, dotsize=0.4, fillcolor=white](0.6013889,-2.4893055)
\psdots[linecolor=black, dotstyle=o, dotsize=0.4, fillcolor=white](0.20138885,-1.6893054)
\psdots[linecolor=black, dotstyle=o, dotsize=0.4, fillcolor=white](0.20138885,-0.8893054)
\psdots[linecolor=black, dotstyle=o, dotsize=0.4, fillcolor=white](0.6013889,-0.08930542)
\psdots[linecolor=black, dotstyle=o, dotsize=0.4, fillcolor=white](1.8013889,-0.4893054)
\psdots[linecolor=black, dotstyle=o, dotsize=0.4, fillcolor=white](2.2013888,0.31069458)
\psdots[linecolor=black, dotstyle=o, dotsize=0.4, fillcolor=white](3.401389,0.71069455)
\psdots[linecolor=black, dotstyle=o, dotsize=0.4, fillcolor=white](4.601389,0.31069458)
\psdots[linecolor=black, dotstyle=o, dotsize=0.4, fillcolor=white](5.001389,-0.4893054)
\psline[linecolor=black, linewidth=0.08](14.601389,-4.4893055)(15.801389,-3.2893054)(15.801389,-3.2893054)
\psline[linecolor=black, linewidth=0.08](14.601389,-4.4893055)(16.20139,-4.0893054)(16.20139,-4.0893054)
\psline[linecolor=black, linewidth=0.08](14.601389,-4.4893055)(16.20139,-4.8893056)(16.20139,-4.8893056)
\psline[linecolor=black, linewidth=0.08](14.601389,-4.4893055)(15.801389,-5.6893053)(15.801389,-5.6893053)
\psline[linecolor=black, linewidth=0.08](11.401389,-4.4893055)(10.201389,-5.6893053)(10.201389,-5.6893053)
\psline[linecolor=black, linewidth=0.08](9.801389,-4.8893056)(11.401389,-4.4893055)(11.401389,-4.4893055)
\psline[linecolor=black, linewidth=0.08](9.801389,-4.0893054)(11.401389,-4.4893055)(11.401389,-4.4893055)
\psline[linecolor=black, linewidth=0.08](10.201389,-3.2893054)(11.401389,-4.4893055)(11.401389,-4.4893055)
\psline[linecolor=black, linewidth=0.08](11.401389,-1.2893054)(10.201389,-0.08930542)(10.201389,-0.08930542)
\psline[linecolor=black, linewidth=0.08](9.801389,-0.8893054)(11.401389,-1.2893054)(11.401389,-1.2893054)
\psline[linecolor=black, linewidth=0.08](9.801389,-1.6893054)(11.401389,-1.2893054)(11.401389,-1.2893054)
\psline[linecolor=black, linewidth=0.08](10.201389,-2.4893055)(11.401389,-1.2893054)(11.401389,-1.2893054)
\psline[linecolor=black, linewidth=0.08](14.601389,-1.2893054)(15.801389,-2.4893055)(15.801389,-2.4893055)
\psline[linecolor=black, linewidth=0.08](14.601389,-1.2893054)(15.801389,-1.6893054)(15.801389,-1.6893054)
\psline[linecolor=black, linewidth=0.08](14.601389,-1.2893054)(15.801389,-0.8893054)(15.801389,-0.8893054)
\psline[linecolor=black, linewidth=0.08](14.601389,-1.2893054)(15.801389,-0.08930542)(15.801389,-0.08930542)
\psline[linecolor=black, linewidth=0.08](14.601389,-1.2893054)(13.001389,-0.4893054)(13.001389,-0.4893054)
\psline[linecolor=black, linewidth=0.08](15.801389,-0.08930542)(17.001389,-0.08930542)(17.001389,-0.08930542)
\psline[linecolor=black, linewidth=0.08](15.801389,-0.8893054)(17.001389,-0.8893054)(17.001389,-0.8893054)
\psline[linecolor=black, linewidth=0.08](15.801389,-1.6893054)(17.001389,-1.6893054)(17.001389,-1.6893054)
\psline[linecolor=black, linewidth=0.08](15.801389,-2.4893055)(17.001389,-2.4893055)(17.001389,-2.4893055)
\psline[linecolor=black, linewidth=0.08](13.001389,-0.4893054)(11.401389,-0.4893054)(13.001389,-0.4893054)(14.601389,-0.4893054)(14.601389,-0.4893054)
\psline[linecolor=black, linewidth=0.08](13.001389,-0.4893054)(14.201389,0.31069458)(14.201389,0.31069458)
\psline[linecolor=black, linewidth=0.08](13.001389,-0.4893054)(11.801389,0.31069458)(11.801389,0.31069458)
\psline[linecolor=black, linewidth=0.08](13.001389,-0.4893054)(13.001389,0.71069455)(13.001389,0.71069455)
\psdots[linecolor=black, dotsize=0.4](11.401389,-4.4893055)
\psdots[linecolor=black, dotsize=0.4](11.401389,-1.2893054)
\psdots[linecolor=black, dotsize=0.4](14.601389,-4.4893055)
\psdots[linecolor=black, dotsize=0.4](15.801389,-0.08930542)
\psdots[linecolor=black, dotsize=0.4](15.801389,-0.8893054)
\psdots[linecolor=black, dotsize=0.4](15.801389,-1.6893054)
\psdots[linecolor=black, dotsize=0.4](15.801389,-2.4893055)
\psdots[linecolor=black, dotsize=0.4](13.001389,-0.4893054)
\psdots[linecolor=black, dotstyle=o, dotsize=0.4, fillcolor=white](17.001389,-0.08930542)
\psdots[linecolor=black, dotstyle=o, dotsize=0.4, fillcolor=white](17.001389,-0.8893054)
\psdots[linecolor=black, dotstyle=o, dotsize=0.4, fillcolor=white](17.001389,-1.6893054)
\psdots[linecolor=black, dotstyle=o, dotsize=0.4, fillcolor=white](17.001389,-2.4893055)
\psdots[linecolor=black, dotstyle=o, dotsize=0.4, fillcolor=white](15.801389,-3.2893054)
\psdots[linecolor=black, dotstyle=o, dotsize=0.4, fillcolor=white](16.20139,-4.0893054)
\psdots[linecolor=black, dotstyle=o, dotsize=0.4, fillcolor=white](16.20139,-4.8893056)
\psdots[linecolor=black, dotstyle=o, dotsize=0.4, fillcolor=white](15.801389,-5.6893053)
\psdots[linecolor=black, dotstyle=o, dotsize=0.4, fillcolor=white](10.201389,-5.6893053)
\psdots[linecolor=black, dotstyle=o, dotsize=0.4, fillcolor=white](9.801389,-4.8893056)
\psdots[linecolor=black, dotstyle=o, dotsize=0.4, fillcolor=white](9.801389,-4.0893054)
\psdots[linecolor=black, dotstyle=o, dotsize=0.4, fillcolor=white](10.201389,-3.2893054)
\psdots[linecolor=black, dotstyle=o, dotsize=0.4, fillcolor=white](10.201389,-2.4893055)
\psdots[linecolor=black, dotstyle=o, dotsize=0.4, fillcolor=white](9.801389,-1.6893054)
\psdots[linecolor=black, dotstyle=o, dotsize=0.4, fillcolor=white](9.801389,-0.8893054)
\psdots[linecolor=black, dotstyle=o, dotsize=0.4, fillcolor=white](10.201389,-0.08930542)
\psdots[linecolor=black, dotstyle=o, dotsize=0.4, fillcolor=white](11.401389,-0.4893054)
\psdots[linecolor=black, dotstyle=o, dotsize=0.4, fillcolor=white](11.801389,0.31069458)
\psdots[linecolor=black, dotstyle=o, dotsize=0.4, fillcolor=white](13.001389,0.71069455)
\psdots[linecolor=black, dotstyle=o, dotsize=0.4, fillcolor=white](14.201389,0.31069458)
\psdots[linecolor=black, dotstyle=o, dotsize=0.4, fillcolor=white](14.601389,-0.4893054)
\psline[linecolor=black, linewidth=0.08](11.401389,-1.2893054)(14.601389,-1.2893054)(14.601389,-4.4893055)(11.401389,-4.4893055)(11.401389,-1.2893054)(14.601389,-4.4893055)(14.601389,-4.4893055)
\psline[linecolor=black, linewidth=0.08](11.401389,-4.4893055)(14.601389,-1.2893054)(14.601389,-1.2893054)
\psdots[linecolor=black, dotstyle=o, dotsize=0.4, fillcolor=white](3.401389,-2.8893054)
\psdots[linecolor=black, dotstyle=o, dotsize=0.4, fillcolor=white](5.001389,-1.2893054)
\psdots[linecolor=black, dotsize=0.4](14.601389,-1.2893054)
\rput[bl](3.2813888,-3.4493055){$v$}
\rput[bl](3.1813889,-5.8493056){\Large{$G$}}
\rput[bl](12.581388,-6.1093054){\Large{$G/v$}}
\end{pspicture}
}
\end{center}
\caption{\small{Graphs $G$ and $G/v$, which  show the tightness of the upper bound of Theorem \ref{thm:G/v-nonpendant}}}\label{fig:G/v-upper}
\end{figure}

	\begin{proof}
First we obtain  the upper bound. Suppose that $D$ is a $\gst$-set of $G$. First suppose that  
$v\in D$. If  $u\in N(v)$ is the vertex with the maximum degree among others, then  $\left( D\cup \{u\}\right)\setminus\{v\}$ is a strong dominating set of $G$, because each vertex is strong dominated  by the same vertex as before or possibly by $u$. 
Now suppose that $v\notin D$.  If  $w\in N(v)$ is the vertex with the maximum degree among others, then by the same argument, $ D\cup \{w\}$ is a strong dominating set of $G$.
Therefore we have $\gst(G/v)\leq \gst (G)+1$. To  show that this bound is tight, consider graph $G$ and $G/v$ in Figure \ref{fig:G/v-upper}. One can easily check that the set of black vertices is a $\gst$-set of both graphs, and therefore $\gst(G/v)=\gst (G)+1$. 
Now, we obtain  the lower bound. To show the lower bound, first we form $G/v$. Suppose that $S$ is $\gst$-set of $G/v$. We remove all the added edges and add $v$ to form $G$. We consider the following cases:
\begin{itemize}
\item[(i)]
$N(v)\subseteq S$. Clearly $S\cup\{v\}$ is a strong dominating set of $G$ and we have $\gst (G)\leq \gst(G/v)+1$.

\item[(ii)]
$N(v)\subseteq V\setminus S$. So  $S\cup\{v\}$ is a strong dominating set of $G$, because each vertex is strong dominated  by the same vertex as before (and possibly $v$). So $\gst (G)\leq \gst(G/v)+1$.

\item[(iii)]
For all vertices  $u\in N(v)$, $\deg_G(v)\geq \deg_G(u)$. So  by the same argument as Case (ii), $S\cup\{v\}$ is a strong dominating set of $G$, and $\gst (G)\leq \gst(G/v)+1$.

\item[(iv)]
There exists a vertex $u\in N(v)\cap S$ such that $\deg_G(u)\geq \deg_G(v)$. So  $S\cup N(v)$ is a strong dominating set of $G$, because  $v$ is strong dominated  by $u$ and the rest of vertices are strong  dominated  by the same vertices as before. So $\gst (G)\leq \gst(G/v)+\deg(v)-1$.

\item[(v)]
There exists a vertex $u\in N(v)\cap\left( V\setminus S\right)$ such that $\deg_G(u)\geq \deg_G(v)$. If $N(v)\subseteq V\setminus S$, then it is Case (ii). So suppose that there exists a vertex $w$ such that $w\in N(v)\cap S$. Then similar to Case (iv), $S\cup N(v)$ is a strong dominating set of $G$, since $v$ is strong dominated  by $u$, and we are done.

\end{itemize}
Hence in general, we have $\gst (G/v)\geq \gst(G)-\deg(v)+1$. Now, we show that this bound is tight. Consider graph $G$ and $G/v$ in Figure \ref{fig:G/v-lower}. The set of black vertices is a $\gst$-set of both graphs, and  $\gst(G/v)=\gst (G)-\deg(v)+1$.
\qed
	\end{proof}

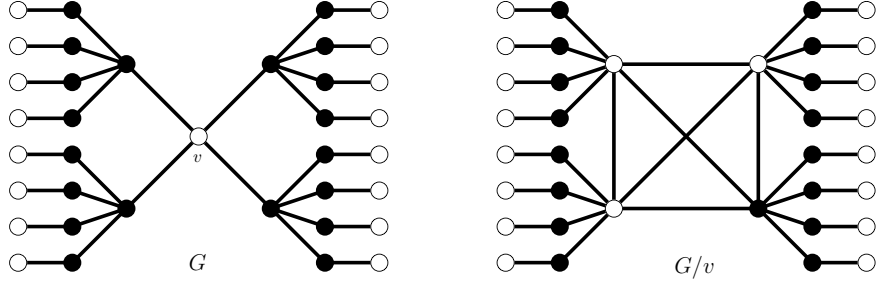
\begin{figure}
\begin{center}
\psscalebox{0.6 0.6}
{
\begin{pspicture}(0,-6.4793057)(19.202778,-0.31791657)
\psline[linecolor=black, linewidth=0.08](4.201389,-3.3193054)(2.601389,-1.7193054)(2.601389,-1.7193054)
\psline[linecolor=black, linewidth=0.08](4.201389,-3.3193054)(5.8013887,-1.7193054)(5.8013887,-1.7193054)
\psline[linecolor=black, linewidth=0.08](4.201389,-3.3193054)(5.8013887,-4.9193053)(5.8013887,-4.9193053)
\psline[linecolor=black, linewidth=0.08](4.201389,-3.3193054)(2.601389,-4.9193053)(2.601389,-4.9193053)
\psline[linecolor=black, linewidth=0.08](5.8013887,-1.7193054)(7.001389,-2.9193053)(7.001389,-2.9193053)
\psline[linecolor=black, linewidth=0.08](5.8013887,-1.7193054)(7.001389,-2.1193054)(7.001389,-2.1193054)
\psline[linecolor=black, linewidth=0.08](5.8013887,-1.7193054)(7.001389,-1.3193054)(7.001389,-1.3193054)
\psline[linecolor=black, linewidth=0.08](5.8013887,-1.7193054)(7.001389,-0.5193054)(7.001389,-0.5193054)
\psline[linecolor=black, linewidth=0.08](7.001389,-0.5193054)(8.201389,-0.5193054)(8.201389,-0.5193054)
\psline[linecolor=black, linewidth=0.08](7.001389,-1.3193054)(8.201389,-1.3193054)(8.201389,-1.3193054)
\psline[linecolor=black, linewidth=0.08](7.001389,-2.1193054)(8.201389,-2.1193054)(8.201389,-2.1193054)
\psline[linecolor=black, linewidth=0.08](7.001389,-2.9193053)(8.201389,-2.9193053)(8.201389,-2.9193053)
\psdots[linecolor=black, dotsize=0.4](7.001389,-0.5193054)
\psdots[linecolor=black, dotsize=0.4](7.001389,-1.3193054)
\psdots[linecolor=black, dotsize=0.4](7.001389,-2.1193054)
\psdots[linecolor=black, dotsize=0.4](7.001389,-2.9193053)
\psdots[linecolor=black, dotstyle=o, dotsize=0.4, fillcolor=white](8.201389,-0.5193054)
\psdots[linecolor=black, dotstyle=o, dotsize=0.4, fillcolor=white](8.201389,-1.3193054)
\psdots[linecolor=black, dotstyle=o, dotsize=0.4, fillcolor=white](8.201389,-2.1193054)
\psdots[linecolor=black, dotstyle=o, dotsize=0.4, fillcolor=white](8.201389,-2.9193053)
\psdots[linecolor=black, dotstyle=o, dotsize=0.4, fillcolor=white](4.201389,-3.3193054)
\rput[bl](4.081389,-3.8793054){$v$}
\rput[bl](3.9813888,-6.2793055){\Large{$G$}}
\psline[linecolor=black, linewidth=0.08](5.8013887,-4.9193053)(7.001389,-6.1193056)(7.001389,-6.1193056)
\psline[linecolor=black, linewidth=0.08](5.8013887,-4.9193053)(7.001389,-5.3193054)(7.001389,-5.3193054)
\psline[linecolor=black, linewidth=0.08](5.8013887,-4.9193053)(7.001389,-4.519305)(7.001389,-4.519305)
\psline[linecolor=black, linewidth=0.08](5.8013887,-4.9193053)(7.001389,-3.7193055)(7.001389,-3.7193055)
\psline[linecolor=black, linewidth=0.08](7.001389,-3.7193055)(8.201389,-3.7193055)(8.201389,-3.7193055)
\psline[linecolor=black, linewidth=0.08](7.001389,-4.519305)(8.201389,-4.519305)(8.201389,-4.519305)
\psline[linecolor=black, linewidth=0.08](7.001389,-5.3193054)(8.201389,-5.3193054)(8.201389,-5.3193054)
\psline[linecolor=black, linewidth=0.08](7.001389,-6.1193056)(8.201389,-6.1193056)(8.201389,-6.1193056)
\psdots[linecolor=black, dotsize=0.4](7.001389,-3.7193055)
\psdots[linecolor=black, dotsize=0.4](7.001389,-4.519305)
\psdots[linecolor=black, dotsize=0.4](7.001389,-5.3193054)
\psdots[linecolor=black, dotsize=0.4](7.001389,-6.1193056)
\psdots[linecolor=black, dotstyle=o, dotsize=0.4, fillcolor=white](8.201389,-3.7193055)
\psdots[linecolor=black, dotstyle=o, dotsize=0.4, fillcolor=white](8.201389,-4.519305)
\psdots[linecolor=black, dotstyle=o, dotsize=0.4, fillcolor=white](8.201389,-5.3193054)
\psdots[linecolor=black, dotstyle=o, dotsize=0.4, fillcolor=white](8.201389,-6.1193056)
\psline[linecolor=black, linewidth=0.08](2.601389,-1.7193054)(1.4013889,-0.5193054)(1.4013889,-0.5193054)
\psline[linecolor=black, linewidth=0.08](1.4013889,-1.3193054)(2.601389,-1.7193054)(2.601389,-1.7193054)
\psline[linecolor=black, linewidth=0.08](1.4013889,-2.1193054)(2.601389,-1.7193054)(2.601389,-1.7193054)
\psline[linecolor=black, linewidth=0.08](1.4013889,-2.9193053)(2.601389,-1.7193054)(2.601389,-1.7193054)
\psline[linecolor=black, linewidth=0.08](1.4013889,-0.5193054)(0.20138885,-0.5193054)(0.20138885,-0.5193054)
\psline[linecolor=black, linewidth=0.08](1.4013889,-1.3193054)(0.20138885,-1.3193054)(0.20138885,-1.3193054)
\psline[linecolor=black, linewidth=0.08](1.4013889,-2.1193054)(0.20138885,-2.1193054)(0.20138885,-2.1193054)
\psline[linecolor=black, linewidth=0.08](1.4013889,-2.9193053)(0.20138885,-2.9193053)(0.20138885,-2.9193053)
\psdots[linecolor=black, dotsize=0.4](1.4013889,-0.5193054)
\psdots[linecolor=black, dotsize=0.4](1.4013889,-1.3193054)
\psdots[linecolor=black, dotsize=0.4](1.4013889,-2.1193054)
\psdots[linecolor=black, dotsize=0.4](1.4013889,-2.9193053)
\psline[linecolor=black, linewidth=0.08](2.601389,-4.9193053)(1.4013889,-3.7193055)(1.4013889,-3.7193055)
\psline[linecolor=black, linewidth=0.08](1.4013889,-4.519305)(2.601389,-4.9193053)(2.601389,-4.9193053)
\psline[linecolor=black, linewidth=0.08](1.4013889,-5.3193054)(2.601389,-4.9193053)(2.601389,-4.9193053)
\psline[linecolor=black, linewidth=0.08](1.4013889,-6.1193056)(2.601389,-4.9193053)(2.601389,-4.9193053)
\psline[linecolor=black, linewidth=0.08](1.4013889,-3.7193055)(0.20138885,-3.7193055)(0.20138885,-3.7193055)
\psline[linecolor=black, linewidth=0.08](1.4013889,-4.519305)(0.20138885,-4.519305)(0.20138885,-4.519305)
\psline[linecolor=black, linewidth=0.08](1.4013889,-5.3193054)(0.20138885,-5.3193054)(0.20138885,-5.3193054)
\psline[linecolor=black, linewidth=0.08](1.4013889,-6.1193056)(0.20138885,-6.1193056)(0.20138885,-6.1193056)
\psdots[linecolor=black, dotsize=0.4](1.4013889,-3.7193055)
\psdots[linecolor=black, dotsize=0.4](1.4013889,-4.519305)
\psdots[linecolor=black, dotsize=0.4](1.4013889,-5.3193054)
\psdots[linecolor=black, dotsize=0.4](1.4013889,-6.1193056)
\psline[linecolor=black, linewidth=0.08](16.601389,-1.7193054)(17.80139,-2.9193053)(17.80139,-2.9193053)
\psline[linecolor=black, linewidth=0.08](16.601389,-1.7193054)(17.80139,-2.1193054)(17.80139,-2.1193054)
\psline[linecolor=black, linewidth=0.08](16.601389,-1.7193054)(17.80139,-1.3193054)(17.80139,-1.3193054)
\psline[linecolor=black, linewidth=0.08](16.601389,-1.7193054)(17.80139,-0.5193054)(17.80139,-0.5193054)
\psline[linecolor=black, linewidth=0.08](17.80139,-0.5193054)(19.001389,-0.5193054)(19.001389,-0.5193054)
\psline[linecolor=black, linewidth=0.08](17.80139,-1.3193054)(19.001389,-1.3193054)(19.001389,-1.3193054)
\psline[linecolor=black, linewidth=0.08](17.80139,-2.1193054)(19.001389,-2.1193054)(19.001389,-2.1193054)
\psline[linecolor=black, linewidth=0.08](17.80139,-2.9193053)(19.001389,-2.9193053)(19.001389,-2.9193053)
\psdots[linecolor=black, dotsize=0.4](17.80139,-0.5193054)
\psdots[linecolor=black, dotsize=0.4](17.80139,-1.3193054)
\psdots[linecolor=black, dotsize=0.4](17.80139,-2.1193054)
\psdots[linecolor=black, dotsize=0.4](17.80139,-2.9193053)
\psdots[linecolor=black, dotstyle=o, dotsize=0.4, fillcolor=white](19.001389,-0.5193054)
\psdots[linecolor=black, dotstyle=o, dotsize=0.4, fillcolor=white](19.001389,-1.3193054)
\psdots[linecolor=black, dotstyle=o, dotsize=0.4, fillcolor=white](19.001389,-2.1193054)
\psdots[linecolor=black, dotstyle=o, dotsize=0.4, fillcolor=white](19.001389,-2.9193053)
\psline[linecolor=black, linewidth=0.08](16.601389,-4.9193053)(17.80139,-6.1193056)(17.80139,-6.1193056)
\psline[linecolor=black, linewidth=0.08](16.601389,-4.9193053)(17.80139,-5.3193054)(17.80139,-5.3193054)
\psline[linecolor=black, linewidth=0.08](16.601389,-4.9193053)(17.80139,-4.519305)(17.80139,-4.519305)
\psline[linecolor=black, linewidth=0.08](16.601389,-4.9193053)(17.80139,-3.7193055)(17.80139,-3.7193055)
\psline[linecolor=black, linewidth=0.08](17.80139,-3.7193055)(19.001389,-3.7193055)(19.001389,-3.7193055)
\psline[linecolor=black, linewidth=0.08](17.80139,-4.519305)(19.001389,-4.519305)(19.001389,-4.519305)
\psline[linecolor=black, linewidth=0.08](17.80139,-5.3193054)(19.001389,-5.3193054)(19.001389,-5.3193054)
\psline[linecolor=black, linewidth=0.08](17.80139,-6.1193056)(19.001389,-6.1193056)(19.001389,-6.1193056)
\psdots[linecolor=black, dotsize=0.4](17.80139,-3.7193055)
\psdots[linecolor=black, dotsize=0.4](17.80139,-4.519305)
\psdots[linecolor=black, dotsize=0.4](17.80139,-5.3193054)
\psdots[linecolor=black, dotsize=0.4](17.80139,-6.1193056)
\psdots[linecolor=black, dotstyle=o, dotsize=0.4, fillcolor=white](19.001389,-3.7193055)
\psdots[linecolor=black, dotstyle=o, dotsize=0.4, fillcolor=white](19.001389,-4.519305)
\psdots[linecolor=black, dotstyle=o, dotsize=0.4, fillcolor=white](19.001389,-5.3193054)
\psdots[linecolor=black, dotstyle=o, dotsize=0.4, fillcolor=white](19.001389,-6.1193056)
\psline[linecolor=black, linewidth=0.08](13.401389,-1.7193054)(12.201389,-0.5193054)(12.201389,-0.5193054)
\psline[linecolor=black, linewidth=0.08](12.201389,-1.3193054)(13.401389,-1.7193054)(13.401389,-1.7193054)
\psline[linecolor=black, linewidth=0.08](12.201389,-2.1193054)(13.401389,-1.7193054)(13.401389,-1.7193054)
\psline[linecolor=black, linewidth=0.08](12.201389,-2.9193053)(13.401389,-1.7193054)(13.401389,-1.7193054)
\psline[linecolor=black, linewidth=0.08](12.201389,-0.5193054)(11.001389,-0.5193054)(11.001389,-0.5193054)
\psline[linecolor=black, linewidth=0.08](12.201389,-1.3193054)(11.001389,-1.3193054)(11.001389,-1.3193054)
\psline[linecolor=black, linewidth=0.08](12.201389,-2.1193054)(11.001389,-2.1193054)(11.001389,-2.1193054)
\psline[linecolor=black, linewidth=0.08](12.201389,-2.9193053)(11.001389,-2.9193053)(11.001389,-2.9193053)
\psdots[linecolor=black, dotsize=0.4](12.201389,-0.5193054)
\psdots[linecolor=black, dotsize=0.4](12.201389,-1.3193054)
\psdots[linecolor=black, dotsize=0.4](12.201389,-2.1193054)
\psdots[linecolor=black, dotsize=0.4](12.201389,-2.9193053)
\psline[linecolor=black, linewidth=0.08](13.401389,-4.9193053)(12.201389,-3.7193055)(12.201389,-3.7193055)
\psline[linecolor=black, linewidth=0.08](12.201389,-4.519305)(13.401389,-4.9193053)(13.401389,-4.9193053)
\psline[linecolor=black, linewidth=0.08](12.201389,-5.3193054)(13.401389,-4.9193053)(13.401389,-4.9193053)
\psline[linecolor=black, linewidth=0.08](12.201389,-6.1193056)(13.401389,-4.9193053)(13.401389,-4.9193053)
\psline[linecolor=black, linewidth=0.08](12.201389,-3.7193055)(11.001389,-3.7193055)(11.001389,-3.7193055)
\psline[linecolor=black, linewidth=0.08](12.201389,-4.519305)(11.001389,-4.519305)(11.001389,-4.519305)
\psline[linecolor=black, linewidth=0.08](12.201389,-5.3193054)(11.001389,-5.3193054)(11.001389,-5.3193054)
\psline[linecolor=black, linewidth=0.08](12.201389,-6.1193056)(11.001389,-6.1193056)(11.001389,-6.1193056)
\psdots[linecolor=black, dotsize=0.4](12.201389,-3.7193055)
\psdots[linecolor=black, dotsize=0.4](12.201389,-4.519305)
\psdots[linecolor=black, dotsize=0.4](12.201389,-5.3193054)
\psdots[linecolor=black, dotsize=0.4](12.201389,-6.1193056)
\rput[bl](14.741389,-6.4793053){\Large{$G/v$}}
\psline[linecolor=black, linewidth=0.08](13.401389,-1.7193054)(16.601389,-1.7193054)(16.601389,-4.9193053)(13.401389,-4.9193053)(13.401389,-1.7193054)(16.601389,-4.9193053)(16.601389,-4.9193053)
\psline[linecolor=black, linewidth=0.08](16.601389,-1.7193054)(13.401389,-4.9193053)(13.401389,-4.9193053)
\psdots[linecolor=black, dotsize=0.4](16.601389,-4.9193053)
\psdots[linecolor=black, dotsize=0.4](5.8013887,-4.9193053)
\psdots[linecolor=black, dotsize=0.4](5.8013887,-1.7193054)
\psdots[linecolor=black, dotsize=0.4](2.601389,-1.7193054)
\psdots[linecolor=black, dotsize=0.4](2.601389,-4.9193053)
\psdots[linecolor=black, dotstyle=o, dotsize=0.4, fillcolor=white](16.601389,-1.7193054)
\psdots[linecolor=black, dotstyle=o, dotsize=0.4, fillcolor=white](13.401389,-1.7193054)
\psdots[linecolor=black, dotstyle=o, dotsize=0.4, fillcolor=white](13.401389,-4.9193053)
\psdots[linecolor=black, dotstyle=o, dotsize=0.4, fillcolor=white](11.001389,-0.5193054)
\psdots[linecolor=black, dotstyle=o, dotsize=0.4, fillcolor=white](11.001389,-1.3193054)
\psdots[linecolor=black, dotstyle=o, dotsize=0.4, fillcolor=white](11.001389,-2.1193054)
\psdots[linecolor=black, dotstyle=o, dotsize=0.4, fillcolor=white](11.001389,-2.9193053)
\psdots[linecolor=black, dotstyle=o, dotsize=0.4, fillcolor=white](11.001389,-3.7193055)
\psdots[linecolor=black, dotstyle=o, dotsize=0.4, fillcolor=white](11.001389,-4.519305)
\psdots[linecolor=black, dotstyle=o, dotsize=0.4, fillcolor=white](11.001389,-5.3193054)
\psdots[linecolor=black, dotstyle=o, dotsize=0.4, fillcolor=white](11.001389,-6.1193056)
\psdots[linecolor=black, dotstyle=o, dotsize=0.4, fillcolor=white](0.20138885,-0.5193054)
\psdots[linecolor=black, dotstyle=o, dotsize=0.4, fillcolor=white](0.20138885,-1.3193054)
\psdots[linecolor=black, dotstyle=o, dotsize=0.4, fillcolor=white](0.20138885,-2.1193054)
\psdots[linecolor=black, dotstyle=o, dotsize=0.4, fillcolor=white](0.20138885,-2.9193053)
\psdots[linecolor=black, dotstyle=o, dotsize=0.4, fillcolor=white](0.20138885,-3.7193055)
\psdots[linecolor=black, dotstyle=o, dotsize=0.4, fillcolor=white](0.20138885,-4.519305)
\psdots[linecolor=black, dotstyle=o, dotsize=0.4, fillcolor=white](0.20138885,-5.3193054)
\psdots[linecolor=black, dotstyle=o, dotsize=0.4, fillcolor=white](0.20138885,-6.1193056)
\end{pspicture}
}
\end{center}
\caption{\small{Graphs $G$ and $G/v$ which show the tightness of the lower bound of Theorem \ref{thm:G/v-nonpendant}}}\label{fig:G/v-lower}
\end{figure}


	\begin{theorem}\label{thm:G/v-pendant}
 Let $G=(V,E)$ be a graph and $v\in V$ is a pendant vertex such that $uv\in E$.  Then,
 $$\gst (G)-1\leq \gst(G/v)\leq \gst (G)+\deg(u)-1.$$
  Furthermore, these bounds are tight.
	\end{theorem}

	\begin{proof}
First we consider the upper bound. Suppose that $D$ is a $\gst$-set of $G$. So  clearly $u\in D$. We have the following cases:
\begin{itemize}
\item[(i)]
$u$ do not strong dominate  other vertices except $v$. Then $D$ is strong dominating set of $G/v$ and we have $\gst(G/v)\leq \gst (G)$.

\item[(ii)]
$u$  strong dominate  $w\neq v$ and $\deg_G(u)\geq\deg(w)+1$. Then $\left( D\cup N(u)\right)\setminus \{w\} $ is strong dominating set of $G/v$ and we have $\gst(G/v)\leq \gst (G)+\deg(u)-1$, because all vertices are strong dominating by the same vertices as before.

\item[(iii)]
$u$  strong dominate  $w\neq v$ and $\deg_G(u)= \deg(w)$. Then $\left( D\cup N(u)\right)\setminus \{u\} $ is strong dominating set of $G/v$ and we have $\gst(G/v)\leq \gst (G)+\deg(u)-1$, because  all vertices are strong dominated  by the same vertices as before, and $u$ is strong dominated  by $w$.

\end{itemize}
Hence $\gst(G/v)\leq \gst (G)+\deg(u)-1$. Now we show that this bound is tight. Consider graph $G$ in Figure \ref{fig:G/v-upper-pendant}. One can easily check that the set of black vertices is a $\gst$-set of $G$, say $S$, and $\left(S\cup N(u)\right)\setminus \{u\}$ is a $\gst$-set of $G/v$. Now, we consider the lower bound. First, we form $G/v$ and find a $\gst$-set of $G/v$, say $D$. Then one can easily check that $D\cup\{u\}$ is a strong dominating set of $G$, and we have $\gst (G)\leq \gst(G/v)+1$. If we consider $G$ as path graph of order $3k+1$, where $k\in\mathbb{N}$, then we see that  $\gst (G/v)= \gst(G)-1$, and the tightness holds.
\qed
	\end{proof}

\begin{figure}
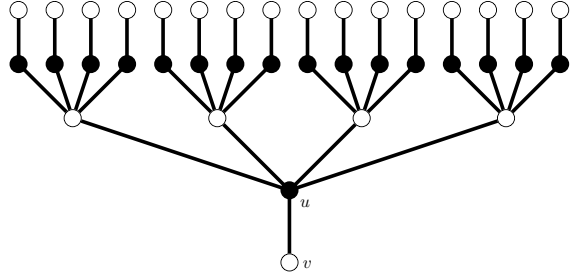

\begin{center}
\psscalebox{0.6 0.6}
{

}
\end{center}
\caption{\small{Graphs $G$ which  show the tightness of the upper bound of Theorem \ref{thm:G/v-pendant}}}\label{fig:G/v-upper-pendant}
\end{figure}

	\begin{figure}
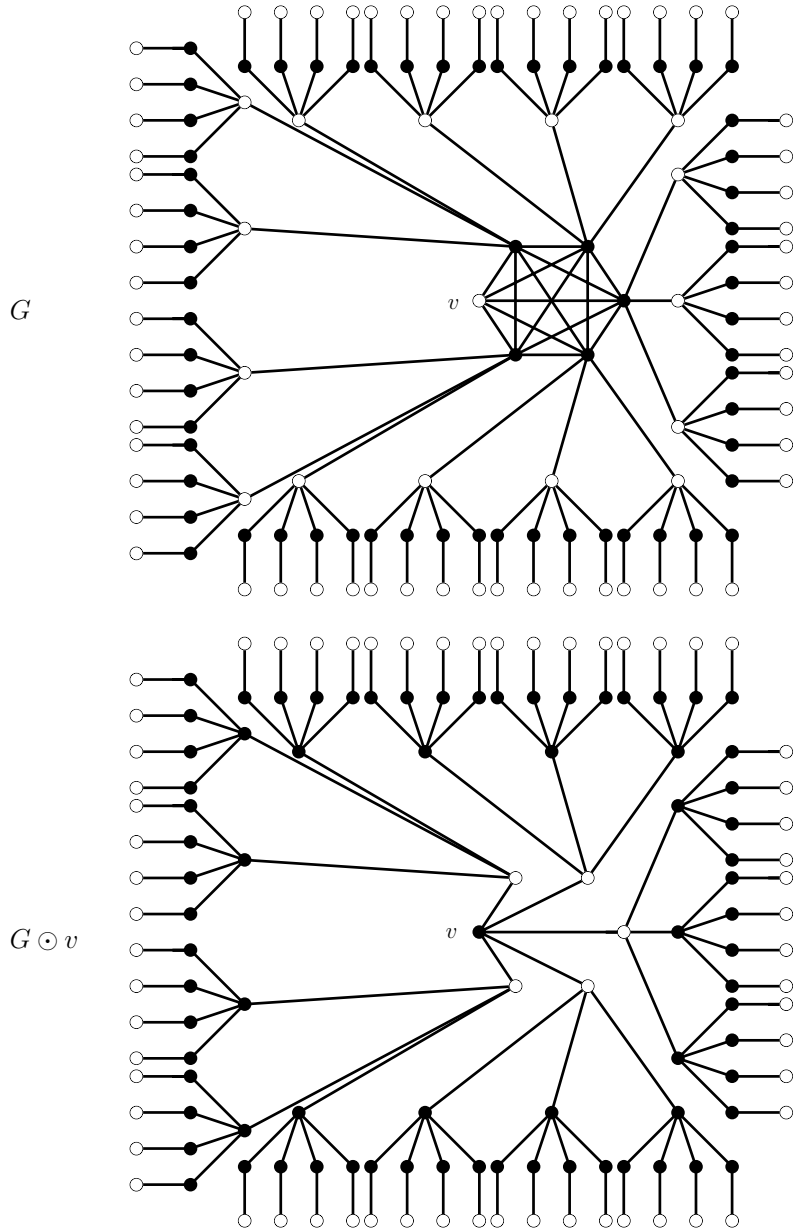

\begin{center}
\psscalebox{0.6 0.6}
{
%
}
\end{center}
\caption{\small{Graphs $G$ and $G\odot v$ which show the tightness of the upper bound of Theorem \ref{thm:Godotv-upper}}}\label{fig:Godotv-upper}
\end{figure}


We have the following result as an immediate outcome of Theorem \ref{thm:G/v-pendant}.	

	\begin{corollary}
 If  $G=(V,E)$ is a graph and $e=uv\in V$ such that $v$ is the pendant vertex, then,
 $$\gst (G)-1\leq \gst(G/e)\leq \gst (G)+\deg(u)-1.$$	
	\end{corollary}


Here we consider a modified graph which is obtained by another operation on a vertex. 	We denote by $G\odot v$ the graph obtained from $G$ by the removal of 
all edges between any pair of neighbors of $v$, note $v$ is not removed from 
the graph. This operation  removes triangles from the graph and for the first time has considered for computation of domination polynomial of a graph, which is the generating function for the number of dominating sets of graphs (\cite{Emeric}).

\begin{theorem}\label{thm:Godotv-upper}
	Let $G=(V,E)$ be a graph and $v\in V$.  Then,
	$$\gst(G\odot v)\leq \gst(G)+1-2\deg_G(v)+\sum_{u\sim v}\deg_{G\odot v}(u).$$
	Furthermore, this bound is  tight.
\end{theorem}

	\begin{proof}
If $v$ is a pendant vertex, then we have nothing to prove, because $\gst(G\odot v)= \gst(G)$. So in the following, suppose that $v$ is not a pendant vertex, and $D$ is a $\gst$-set of $G$. 
If $x\in N(v)\cap D$, and does not strong dominate other vertices, then we simply keep it for strong dominating set of  $G\odot v$, and add $v$ to $D$. If $x\in N(v)\cap \left(V\setminus D\right)$, then we add $x$ and add $v$ to $D$. So, in every cases,  all  $x\in N(v)$ are strong dominate  some other vertices, and these vertices are not in $N(v)$. Suppose that $x$ is strong dominate  $y$ and $y\notin N(v)$. If after forming $G\odot v$, $\deg_{G\odot v}(x)\geq \deg(y)$, then we just add $v$ to $D$. But sometimes, we need to add all neighbours of $x$ to $D$ and remove $x$ from $D$. So in this  case, 
The set 
$$D'=\left(D \setminus N(v) \right)\cup \{v\} \bigcup_{u\sim  v}\left(N_{G\odot v}(u)\setminus \{v\}\right),$$ 
is a strong dominating set of $G\odot v$ with the biggest size other than what ever we mentioned before, and we have 
$$\gst(G\odot v)\leq \gst(G)+1-\deg_G(v)+\sum_{u\sim v}\left(\deg_{G\odot v}(u)-1\right),$$
and we are done. Now, we show that this bound is tight. Consider graph $G$ and $G\odot v$ in Figure \ref{fig:Godotv-upper}. One can easily check that the set of black vertices is a $\gst$-set of both graphs, and therefore the equality holds.
\qed
	\end{proof}


\end{document}